\documentclass[12pt]{amsart}

\usepackage{amssymb}

\usepackage{enumerate}

\usepackage{graphicx}

\makeatletter
\@namedef{subjclassname@2010}{%
  \textup{2010} Mathematics Subject Classification}
\makeatother

\newtheorem{thm}{Theorem}[section]
\newtheorem{cor}[thm]{Corollary}
\newtheorem{lem}[thm]{Lemma}
\newtheorem{prop}[thm]{Proposition}

\theoremstyle{definition}

\newtheorem{rem}[thm]{Remark}
\newtheorem{exa}[thm]{Example}

\numberwithin{equation}{section}

\newcommand{\C}{\mathbb{C}}
\newcommand{\R}{\mathbb{R}}
\newcommand{\N}{\mathbb{N}}

\newcommand{\mfs}{\mathfrak{s}}

\begin{document}

%%%%% To ease editing, for IMPAN journals add:

\baselineskip=17pt

%%%%%%%%%%%

%% In the running head, replace first names by initials 
%% and give an abbreviation of the title.

\title[New limit theorems related to free multiplicative convolution]{New limit theorems \\
related to free multiplicative convolution}

\author[N. Sakuma]{Noriyoshi Sakuma}
\address{
Department of Mathematics,
Aichi University of Education,
1 Hirosawa, Igaya-cho, Kariya-shi, 448-8542, Japan}
\email{sakuma@auecc.aichi-edu.ac.jp}
\author[H. Yoshida]{Hiroaki Yoshida}
\address{
Department of Information Sciences, 
Ochanomizu University,
2-1-1, Otsuka, Bunkyo, Tokyo 112-8610 Japan}
\email{yoshida@is.ocha.ac.jp}

\date{}

\begin{abstract}
Let $\boxplus$, $\boxtimes$, and $\uplus$ be the free additive, 
free multiplicative, and boolean additive convolutions, respectively.
For a probability measure $\mu$ on $[0,\infty)$ with finite second moment,
we find a scaling limit of $(\mu^{\boxtimes N})^{\boxplus N}$ as $N$ goes to infinity.
The $\mathcal{R}$--transform of its limit distribution can be represented by
the Lambert's $W$ function. 
From this, we prove that the limiting distribution is freely infinitely divisible 
as well as the lognormal distribution in classical sense.
We also show a similar limit theorem by replacing 
the free additive convolution with the boolean convolution.  
\end{abstract}



\maketitle

\section{Introduction}
In probability theory, limit theorems and infinite divisibility are considered in various situations.
The classical references are the books by Gnedenko and Kolmogorov \cite{KoGn68} and Petrov \cite{Pe95}.
One of the most famous limit theorems is the Central Limit Theorem (for short CLT)
that is the scaling limit of the sum of independent, identically distributed (i.i.d.) random variables.
Suppose that a random variable $Z$ has the standard normal distribution.
Let $\{X_{k}\}_{k=1}^{\infty}$ be a sequence of i.i.d. random variables with finite second moment.
Then a scaling 
\begin{equation}\label{scaling}
\frac{X_{1}+ \dots +X_{N}-N\mathbb{E}[X_{1}]}{\sqrt{N\mathbb{V}[X_{1}]}}
\end{equation}
converges to $Z$ in distribution as $N$ goes to infinity.

When we consider the product of i.i.d. random variables, 
we have also a CLT type limit theorem.
The simplest case is as follows:
for a sequence of i.i.d. random variables $\{X_{k}\}_{k=1}^{\infty}$ with finite second moment,
we consider a scaling
\begin{align}\label{prodscale}
\prod_{k=1}^{N}\exp\left( 
\frac{X_{k}-\mathbb{E}[X_{k}]}{\sqrt{N\mathbb{V}[X_{k}]}} 
\right).
\end{align}
By the CLT, this scaling converges to $e^{Z}$ in distribution as $N$ goes to infinity.
The distribution of $e^{Z}$ is called the lognormal distribution.
It was proved by Thorin \cite{Th77b} that the lognormal distribution is infinitely divisible.
The product limit theorems are also interested from applications to statistics. 
For details, see \cite{ReWe02} and the book by Galambos and Simonelli \cite{Gala04}. 

In free probability theory, 
some limit theorems are known as in classical probability theory.
The most famous limit theorem is the free CLT,
which was found by Voiculescu.
If $\{X_k\}_{k\in \N}$ is a sequence of freely independent identically distributed 
(for short freely i.i.d.) random variables with finite second moment,
then the normalized sum
\eqref{scaling} 
converges to the standard Wigner's semi-circle law in distribution as $N$ goes 
to infinity.
In addition, we know the Poisson limit theorem, the stable limit theorem and so on,
for details, see \cite{HiPe00}, \cite{BeVo93}, and \cite{BePa99}. 
Recently other new limit theorems with respect to the free convolutions \cite{BeWa08}, \cite{Wa10}, and \cite{Tu10} have been studied. 

In this paper, we shall prove a limit theorem involving not only free additive 
but also free multiplicative convolutions.
We introduce a new normalized sum of multiplications of freely independent random variables.
For double sequence of freely i.i.d. random variables 
$\{\{X^{(j)}_{i}\}_{i\in\N}\}_{j\in\N}$ having a distribution $\mu$ on $[0,\infty)$
with finite second moment,
we consider a new normalization $Y_{N}$,
	\begin{align}\label{scalefree}
	&Y_{N}=\sum_{j=1}^{N}\frac{
	\sqrt{X^{(j)}_{N}} \cdots 
	\sqrt{X^{(j)}_{2}}
	X^{(j)}_{1}
	\sqrt{X^{(j)}_{2}}
	\cdots 
	\sqrt{X^{(j)}_{N}}}
	{m_{1}^{N}N},&
	\end{align}
where $m_{1}$ is the mean of the distribution $\mu$.
We shall see that
its limit distribution depends only on the first and second moments.
In its proof,
we shall investigate the Taylor type expansion of 
the $S$--transform.
In addition, a formula by Belinschi and Nica \cite{BeNi09} suggests that
the distribution of \eqref{scalefree} is equal to  the one of
$$
\widetilde{Y_{N}} = \frac{\sqrt{\sum_{i=1}^{N}X^{(N)}_{i}}\cdots\sqrt{\sum_{i=1}^{N}X^{(1)}_{i}}\sqrt{\sum_{i=1}^{N}X^{(1)}_{i}}\dots\sqrt{\sum_{i=1}^{N}X^{(N)}_{i}}}{m_{1}^{N}N^{N}},
$$
which is corresponding to the scaling \eqref{prodscale}. 
In this meaning, we may call it free lognormal distribution.
Compare to free additive CLT case, it is not exactly the same scaling. 
The difference may occur because of non-commutativity of random variables.
Furthermore a similar limit theorem can be found under boolean independence.

In order to investigate properties of this limit distribution, we show that it is freely infinitely divisible
as in classical case lognormal distribution is infinitely divisible.
In its proof, the properties of Lambert's $W$-- function play an important role and we obtain its L\'evy measure.

This paper is organized as follows.
In Section $2$, we shall gather the tools for free and boolean probability.
Especially, we recall $\mathcal{R}$, $S$, and $\Sigma$--transforms and infinite divisibility in free probability theory.
In Section $3$, we shall give the Taylor type expansions for $S$ and $\Sigma$--transforms and prove our limit theorems.
Finally, in Section $4$, we shall discuss the limit distributions 
with focusing on infinite divisibility and moments.

%%%%%%%%%%%%%%%%%%%%%%%%%%%%%%%%%%%%%%%
\section{Preliminaries}
Let $\R_+$ be the half line $[0, +\infty)$ and $\C^{+}$ be the upper half plane $\{z=x+iy\in\C;y>0\}$.
We fix notation that
$\mathcal{P}$ and $\mathcal{P}_+$ mean the set of all Borel probability measures
on $\R$ and $\R_+$, respectively.
We denote the free additive, free multiplicative, 
and boolean additive
convolutions 
by $\boxplus$, $\boxtimes$, and $\uplus$, respectively,
see for details of convolutions, \cite{SW97},
\cite{VoDyNi92}, and \cite{NiSp06}.
Hereafter, $\delta_{0}$ stands the Dirac probability measure concentrated on $0$.

%%%%%%%%%%%%%%%%%%%%%%%%%%%%%%%%%%%%%%%%%%%%%%%%%%
\subsection{Analytic tools for free and boolean convolutions}

Here, we shall gather the analytic tools for free and boolean probability and mention some of their important facts.

We denote the Cauchy transform of a probability measure $\mu$ on $\R$ by
\[
G_{\mu}(z) = \int_{\R}\frac{1}{z-x}\mu(dx), \quad z \in \C^+ ,
\]
and 
\[
\Psi_{\rho}(z) = \int_{\R}\frac{xz}{1-xz}\rho(dx), \quad z \in \C \backslash \R
\] 
denotes
the moment generating function of $\rho$ on $\R_+$.
Then the Speicher's $R$ and Voiculescu's $\mathcal{R}$--transforms of $\mu$ are defined as follow:
for any given $\alpha>0$, one can find $\beta>0$ so that 
\[
R_{\mu}(z)= z \mathcal{R}_{\mu}(z)
= z G_{\mu}^{-1}(z)-1, \quad 1/z \in \Gamma_{\alpha,\beta},
\]
where $G_{\mu}^{-1}(z)$ is the right inverse of $G_{\mu}(z)$ with respect to 
composition and 
$\Gamma_{\alpha,\beta}=\{ z=x+iy \in \C^{+}; y>\beta, |y| >\alpha x\}$.
Note that we will use both $R$ and $\mathcal{R}$--transforms for convenience.
The $S$ and $\Sigma$--transforms of $\rho$ are defined by
\[
S_{\rho}(z)=\frac{(z+1)\Psi_{\rho}^{-1}(z)}{z}, \quad z \in \Psi_{\rho}(i\C^+)
\]
and 
\[
\Sigma_{\rho}(z)=S_{\rho}\left(\frac{z}{1-z}\right),
\quad \frac{z}{1-z} \in \Psi_{\rho}(i\C^+),
\]
respectively, where $\Psi_{\rho}^{-1}(z)$ is the right inverse of $\Psi_{\rho}(z)$
with respect to composition.
Now, we summarize the relations between the transforms and convolutions, see for proofs \cite{BeVo93} and \cite{BeNi09}.
\begin{prop}\label{basic}
	For $\mu_1\in\mathcal{P}$, $\mu_2\in\mathcal{P}$, $\rho_1\in\mathcal{P}_+$ and 
	$\rho_2\in\mathcal{P}_+$, which are not $\delta_{0}$, there exist $\alpha>0$ and $\beta>0$ such that
		\begin{align*}
		&R_{\mu_1 \boxplus \mu_2} (z)=
		R_{\mu_1}(z)+R_{\mu_2} (z),
		\quad 1/z \in \Gamma_{\alpha,\beta},&\\
		&S_{\rho_1 \boxtimes \rho_2}(z)
		=S_{\rho_1}(z)S_{\rho_2}(z),\quad z \in \Psi_{\rho_1}(i\C^+)\cap\Psi_{\rho_2}(i\C^+),&\\
		&S_{\rho_1^{\boxplus t}}(z)
		=\frac{1}{_{\,}t_{\,}} S_{\rho_1}\left(\frac{z}{_{\,}t_{\,}}\right),&\\
		&\Sigma_{\rho_1 \boxtimes \rho_2}(z)
		=\Sigma_{\rho_1}(z)\Sigma_{\rho_2}(z),\quad z/(1-z) \in \Psi_{\rho_1}(i\C^+)\cap\Psi_{\rho_2}(i\C^+), &\\
		&\Sigma_{\rho_1^{\uplus t}}(z)
		=\frac{1}{_{\,}t_{\,}} \Sigma_{\rho_1}\left(\frac{z}{_{\,}t_{\,}}\right).&
		\end{align*}
\end{prop} 
For $c > 0$, the dilation operator $D_c$ on $\mathcal{P}$ is defined by 
	\[
	D_c(\mu) (B)= \mu\left(\frac{1}{_{\,}c_{\,}}B\right) 
	\]
for any Borel set $B$ on $\R_+$,
where $\frac{1}{\, c \,}B=\{x\in\R ; \frac{1}{\,c\,}x \in B \}$.
If a random variable $X$ has a distribution $\mu$, then $cX$ is distributed as $D_{c}(\mu)$.
In the paper \cite{BeNi09}, the authors showed that
	\[
	S_{D_c(\mu)}(z) = \frac{1}{_{\,}c_{\,}}S_{\mu}(z)
	\]
and
	\[
	\Sigma_{D_c(\mu)}(z) = \frac{1}{_{\,}c_{\,}}\Sigma_{\mu}(z).
	\]
	
%%%%%%%%%%%%%%%%%%%%%%%%%%%%%%%%%%%%%%%%
\subsection{Infinite divisibility for free additive convolution}
A probability measure $\mu$ is freely infinitely divisible (or $\boxplus$--infinitely divisible) if 
for any $n\in\N$ there exists $\mu_n\in\mathcal{P}$ such that
	\begin{align*}
	\mu = 
	\underbrace{\mu_n \boxplus\cdots \boxplus\mu_n}_{n \text{ times}}.
	\end{align*}
We denote the class of all $\boxplus$--infinitely divisible distributions by $I^{\boxplus}$.
	\begin{rem}
	We can define other infinite divisibility replacing $\boxplus$ 
	by $\boxtimes$ or $\uplus$.
	But for boolean convolution, all probability measures are 
	$\uplus$--infinitely divisible. 
	So we shall not discuss $\uplus$--divisibility any longer.
	\end{rem}
The next proposition characterizes the $\boxplus$--infinitely divisible laws \cite[Theorem3.7.2]{VoDyNi92}. 
\begin{prop}\label{id}
	The followings are equivalent:
		\begin{enumerate}[{\rm (1)}]
		\item $\mu \in I^{\boxplus}$.
		\item $\mathcal{R}_{\mu}$ has an analytic extension defined on $\C^-$ with value $\C^- \cup \R$.
		\item There exist unique $b_{\mu}\in\R$ 
		and finite measure $\nu_{\mu}$
		such that
			\begin{align*}
			\mathcal{R}_{\mu}(z) 
			= b_{\mu} + \int_{\R} \frac{z}{1-tz}\nu_{\mu}({\mathrm d}t),
			\quad z \in \C^-.
			\end{align*}
		\end{enumerate}
\end{prop}
The above expression is called $\boxplus$--L\'evy--Khintchine representation, 
or simply L\'evy--Khintchine representation.

	\begin{exa}The
	typical examples of $\boxplus$--infinitely divisible distribution are
	the Wigner's semi-circle law, the Dirac's delta distribution, and
	the free Poisson distribution $\pi_t$ with parameter $t\geq0$  
	having density
	\begin{equation}
	\pi_{t}(\mathrm{d}x)
	=\max(0,(1-t))\delta_{0}(\mathrm{d}x)
	+\frac{1}{2\pi x}\sqrt{4t-(x-1-t)^{2}}\,\,\mathbb{1}
	_{[(1-\sqrt{t})^{2},(1+\sqrt{t})^{2}]}(x)\mathrm{d}x.  \label{MaPadis}
	\end{equation}
	The L\'evy measure $\nu_{\mu}$ and $b_{\mu}$ 
	of the semi-circle law are $\delta_{0}$ and $0$,
	and the free Poisson law $\pi_t$ has $b_{\mu}=t$ and 
    $\nu_{\mu}=t \delta_{1}$.
	We put $\pi_{1}$ by $\pi$.
	\end{exa}

The following functional equation of the $R$ and $S$--transforms can be found
in, for instance, \cite{NiSp06} or \cite[Lemma 2]{PAS09}:
\begin{prop}\label{Scum}
	Assume that $\mu\in\mathcal{P}_+$. 
	For some sufficiently small $\varepsilon>0$, 
	we have a region $D_{\varepsilon}$ which includes 
    $\{ -it ; 0 < t <\varepsilon \}$ 
	such that
		\begin{align}
		z = R_{\mu} \left(zS_{\mu}(z)\right),
		\end{align}
	for $z\in D_{\varepsilon}$.
\end{prop}
 
%%%%%%%%%%%%%%%%%%%%Section 3%%%%%%%%%%%%%%%%%%%%%%%%%%
\section{New limit theorems}
In this section, 
we prove a new limit theorem related 
to both free additive and multiplicative convolutions.
We also discuss a similar result with replacing $\boxplus$ by $\uplus$.
It was proved in \cite{Ml10} by M\l otkwski 
that for the free Poisson law $\pi$, we have
	\[
	D_{n}\left(\left(\pi^{\boxtimes (n-1)}\right)^{\uplus n}\right) 
	\overset{n\to \infty}{\longrightarrow} \nu_0 
    \quad \mbox{in distribution},
	\]
where the $p^{\scriptsize \mbox{th}}$ moment of $\nu_0$ is given by $\dfrac{p^p}{p!}$.
We find that a theorem of this type holds more generally if we replace $\pi$ 
by any probability distribution with finite second moment.

%%%%%%%%%%%%%%%%%%%%%%%%%%%%%%%%%%%%%%%%%%%%%%%%%
\subsection{Expansion of the $S$-transform and $\Sigma$-transform}
We prove the expansion for the $S$--transform and $\Sigma$--transform 
under the second moment condition.
For the $\mathcal{R}$--transform, the Taylor type expansion was proved by Benaych-Georges 
in \cite{BG06}. 
For each region $A$ in $\C$,
we denote $z \to 0$ with $z \in A$
by $z \overset{z \in A}{\longrightarrow} 0$.
\begin{lem}\label{lemM}
	Let $\rho \in \mathcal{P}_+$ have the moment of order $p$, that is,
	for $k=0,1,2,\ldots ,p$,
		\[
		m_{k}(\rho):=\int_{\R_+} x^k \rho (dx) < \infty.
		\]
	Then its moment generating function $\Psi_{\rho}(z)$ has a Taylor expansion
		\begin{align*}
		\Psi_{\rho}(z) = \sum_{k=1}^{p-1} m_{k}(\rho) z^k + O(z^{p}), 
		\quad z \overset{z \in i\C^+}{\longrightarrow} 0.
		\end{align*} 
\end{lem}
\begin{proof}
See \cite{Ak69}.
\end{proof}
\begin{lem}\label{lem}
	Let $\rho \in \mathcal{P}_+$ have the moment of order $p\geq 2$
	 and $\rho \not = \delta_0$.
	 Then we have the followings:
	\begin{enumerate}
		\item[\rm{(1)}] $\Psi_{\rho}(z)$ is univalent in $i\C^+$.
		\item[\rm{(2)}] The inverse function 
	$\Psi^{-1}_{\rho} : \Psi_{\rho}(i\C^+) \to i\C^+$ 
	of $\Psi_{\rho}$ 
	admits Taylor type expansion of order 2
	\begin{align*}
		\Psi^{-1}_{\rho} (z) = \frac{1}{m_1(\rho)} z -\dfrac{m_{2}(\rho)}{(m_{1}(\rho))^{3}} z^{2} + o(z^2),  \quad
		z \overset{z\in \mathfrak{D}_{\rho}}{\longrightarrow} 0.
	\end{align*}
		\item[\rm{(3)}] $\mathfrak{D}_{\rho}:=\Psi_{\rho}(i\C^+)$ is a region contained in the circle with diameter
		$\big(\,\rho(\{0\})-1, \,0\,\big)$.
	In addition,  $\Psi_{\rho}(i\C^+) \cap \R= \big(\,\rho(\{0\})-1,\, 0 \,\big)$,
	\[
	\lim_{t\uparrow 0}\Psi^{-1}_{\rho}(t)=0
	\] 
	and 
	\[
	\lim_{t\downarrow \rho(\{0\})-1}\Psi^{-1}_{\rho}(t)=\infty.
	\]
	\end{enumerate}
\end{lem}
\begin{proof}
(1) and (3) are proved in Bercovici and Voiculescu \cite[Proposition 6.2]{BeVo93}.

(2) \underline{{\bf Step 1}} 
We shall first prove that
\begin{align*}
\Psi^{-1}_{\rho} (z) = \frac{1}{m_1(\rho)} z + o(z),  \quad
z \overset{z\in \mathfrak{D}_{\rho}}{\longrightarrow} 0.
\end{align*}
Take any continuous path $\{z(t)\}_{t\in (0,1]}$ 
in $\mathfrak{D}_{\rho}$ such that $\lim_{t\downarrow 0}z(t)=0$.
By (1),
we can choose a unique continuous path $\{\omega(t)\}_{t\in[0,1]}$
on $t\in (0,1]$
such that 
$\lim_{t\downarrow 0}\omega(t)=0$
and
$\Psi_{\rho}(\omega(t))=z(t)$.
\begin{align*}
\lim_{t\downarrow 0}\frac{\Psi^{-1}_{\rho}(z(t))}{z(t)} 
= \lim_{t\downarrow 0} \frac{\omega(t)}{\Psi_{\rho}(\omega(t))}
= \lim_{t\downarrow 0} \frac{1}{\Psi_{\rho}(\omega(t))/\omega(t)}
= \frac{1}{m_1}.
\end{align*}
By arbitrary of the paths, it follows that
\begin{align*}
\Psi^{-1}_{\rho} (z) = \frac{1}{m_1(\rho)} z + o(z),  \quad
z \overset{z\in \mathfrak{D}_{\rho}}{\longrightarrow} 0.
\end{align*}

\underline{{\bf Step 2}} 
Using step $1$, we shall show the Taylor type expansion of order $2$
as $z \overset{z\in \mathfrak{D}_{\rho}}{\longrightarrow} 0$.
\begin{align*}
& \quad\quad \frac{\Psi_{\rho}^{-1}(\Psi_{\rho} (z)) - \frac{1}{m_1(\rho)}\Psi_{\rho} (z)}{\Psi_{\rho} (z)^{2}}
= \frac{z-\frac{1}{m_1(\rho)}(m_1(\rho)z + m_{2}(\rho)z^{2} + O(z^{3}))}{(m_1(\rho)z + m_{2}(\rho)z^{2} + O(z^{3}))^{2}}&\\
&=\frac{\left(\frac{m_{2}(\rho)}{m_1(\rho)}z^{2} + O(z^{3})\right)}{(m_1(\rho))^{2}z^{2} + O(z^{3})}
\to \frac{m_{2}(\rho)}{(m_1(\rho))^{3}} , 
\quad \text{as }z \overset{z\in \mathfrak{D}_{\rho}}{\longrightarrow} 0. &
\end{align*}
As a results, we obtain as follow:
\begin{align*}
\Psi^{-1}_{\rho} (z) = \frac{1}{m_1(\rho)} z + \frac{m_{2}(\rho)}{(m_1(\rho))^{3}} z^{2}+o(z^{2}),  \quad
z \overset{z\in \mathfrak{D}_{\rho}}{\longrightarrow} 0.
\end{align*}
\end{proof}
%%%%%%%%%%%%%%%%%%%%%%%%%%%%%%%%%%%%%%%%%%
\subsection{Limit theorems}
Here we shall state the main theorem.
\begin{thm}\label{maintheorem}
We assume that 
$\rho \in \mathcal{P}_+$ has the second moment
and put
$\gamma = \frac{\mathrm{Var}(\rho)}{(m_{1}(\rho))^{2}}$.
\begin{itemize}
\item[\rm{(1)}] 
There exist $s_0>0$ and $s_1<0$ such that
the $S$--transform of $\rho$ is given by
\begin{align*}
S_{\rho}(z) = s_0 +  s_{1}z + o(z),   \quad
z \overset{z\in \mathfrak{D}_{\rho}}{\longrightarrow} 0,
\end{align*} 
and there exists a probability measure 
$\mathfrak{y}_{\gamma} \in \mathcal{P}_+$ 
such that
\begin{align*}
D_{s_0^n/n}\left(\left(\rho^{\boxtimes n}\right)^{\boxplus n}\right) 
\to \mathfrak{y}_{\gamma} \quad \mbox{{\rm in distribution}}.
\end{align*}
In addition, the $S$--transform of the limit distribution $\mathfrak{y}_{\gamma}$ is
\begin{align*}
S_{\mathfrak{y}_{\gamma}}(z) = \exp\left(-\gamma  z \right).
\end{align*} 

\item[\rm{(2)}]
There exist $\sigma_0>0$ and $\sigma_1<0$ such that
the $\Sigma$--transform of $\rho$ is given by
\begin{align*}
\Sigma_{\rho}(z) = \sigma_0 +  \sigma_{1}z + o(z),  \quad
z \overset{z\in \mathfrak{D}_{\rho}}{\longrightarrow} 0,
\end{align*} 
and there exists a probability measure 
$\mathfrak{s}_{\gamma} \in \mathcal{P}_+$ 
such that
\begin{align*}
D_{s_0^n/n}\left(\left(\rho^{\boxtimes n-1}\right)^{\uplus n}\right) 
\to \mathfrak{s}_{\gamma}  \quad \mbox{{\rm in distribution}}.
\end{align*}
In addition, the $\Sigma$--transform of the limit distribution $\mathfrak{s}_{\gamma}$ is
\begin{align*}
\Sigma_{\mathfrak{s}_{\gamma}}(z) = \exp\left(-\gamma  z \right).
\end{align*}
\end{itemize} 
\end{thm}
%%%%%%%%%%%%%%%%%%%%%%%%%%%%%%%%%%%%%%
\begin{proof}
Using Lemma \ref{lem}, we have
$$
S_{\rho}(z) = \frac{z+1}{z} \Psi^{-1}_{\rho}(z)=\frac{1}{m_1(\rho)}  - \frac{\mathrm{Var}(\rho)}{(m_1(\rho))^{3}} z+o(z),  \quad
z \overset{z\in \mathfrak{D}_{\rho}}{\longrightarrow} 0.
$$
Let $s_{0} = \frac{1}{m_1(\rho)}$ and $s_{1}=- \frac{\mathrm{Var}(\rho)}{(m_1(\rho))^{3}}$.
By Proposition \ref{basic}, we obtain
\begin{align*}
S_{D_{s_0^n/n}\left(\left(\rho^{\boxtimes n}\right)^{\boxplus n}\right)}(z) 
&=\frac{n}{s_{0}^{n}}S_{ (\rho^{\boxtimes n})^{\boxplus n}}(z)
=\frac{1}{s_{0}^{n}}S_{\rho^{\boxtimes n}}\left(\frac{z}{n}\right)&\\
&=\frac{1}{s_{0}^{n}}\left( S_{\rho}\left(\frac{z}{n}\right)\right)^{n}
=\frac{1}{s_{0}^{n}}\left(  s_{0} + s_{1}\frac{z}{n} +o\left(\frac{1}{n}\right)   \right)^{n}&\\
&=\left( 1 + \frac{s_1 z}{s_{0} n} + o\left(\frac{1}{n}\right) \right)^n &\\
&\to \exp\left(\frac{s_1}{s_0} z \right)= \exp\left(-\gamma  z \right)
\quad\text{as }n\to\infty.&
\end{align*} 
From \cite[Lemma 7.1]{BeVo92} and \cite[Theorem 6.13 (ii)]{BeVo93}, there exists a free multiplicative infinitely divisible measure $\mathfrak{y}_{\gamma}$ 
such that $S_{\mathfrak{y}_{\gamma}}(z) = \exp\left(-\gamma z \right)$.
The proof for (2) is the same as for the free additive case.
\end{proof}
We can exchange the order of free multiplicative and freely additive (or boolean additive)
convolutions. The difference is in the scaling speed.
\begin{cor}
Under the same setting as in Theorem \ref{maintheorem},
we have
\begin{align*}
&D_{s_0^n/n^{n}}\left(\left(\rho^{\boxplus n}\right)^{\boxtimes n}\right) 
\to \mathfrak{y}_{\gamma},&\\
&D_{s_0^n/n^{n}}\left(\left(\rho^{\uplus n-1}\right)^{\boxtimes n}\right) 
\to \mathfrak{s}_{\gamma},&
\end{align*}
as $n \to \infty$.
\end{cor}
\begin{proof}
As we have done in the proof of Theorem \ref{maintheorem},
it can be proved by using Proposition \ref{basic} and Lemma \ref{lem}.
\end{proof}
%%%%%%%%%NEW SECTION%%%%%%%%%%%%%%%%%%%
\section{Lambert $W$ function and infinite divisibility of the limit distribution}
\subsection{On the limit distribution of free case}
When we calculate the $R$--transform or the moment generating function, 
the Lambert's $W$--function plays an important role, which
satisfies the functional equation
\[
z=W(z)\exp (W(z)).
\]
This function have been studied for a long period and we have known several good properties of this function as real and complex function.
For more details of the Lambert  $W$ function, see, 
for instance, \cite{Cetc96}.
Let $W_{0}(z)$ be the principal branch of the Lambert $W$--function.

By Proposition \ref{Scum} and the $S$-transform of $\mathfrak{y}_{\gamma}$, 
we have 
\[
R_{\mathfrak{y}_{\gamma}}(z\mathrm{e}^{-\gamma z})=z , \quad 1/z \in \Gamma_{\alpha,\beta}.
\]
This functional equation suggests that the $R$--transform is given by using the Lambert's $W$--function.

\begin{thm}
\begin{enumerate}
	\item[\rm{(1)}] The $\mathcal{R}$ and $R$--transforms
	of probability measure $\mathfrak{y}_{\gamma}$ 
	are given as follows:
	\begin{align*}
	&\mathcal{R}_{\mathfrak{y}_{\gamma}}(z) 
	=\frac{-W_{0}(-\gamma z)}{\gamma z},&\\
	&R_{\mathfrak{y}_{\gamma}}(z) = -\frac{1}{\gamma}W_{0}(-\gamma z) .&
	\end{align*}

	\item[\rm{(2)}]
	$\mathfrak{y}_{\gamma}$ is both $\boxplus$--infinitely divisible and $\boxtimes$--infinitely divisible.
	\item[\rm{(3)}]
	The free cumulant sequence of 
	$\mathfrak{y_{\gamma}}$ is $\left\{ \dfrac{(\gamma n)^{n-1}}{n!}\right\}_{n\in\N}$.

	\item[\rm{(4)}]
	
	The L\'evy measure $\nu_{\mathfrak{y}_{\gamma}}$ 
	of $\mathfrak{y}_{\gamma}$ is given by
		\[
		\nu_{\mathfrak{y}_{\gamma}}(ds) 
		=\frac{1}{\gamma\pi}sf^{-1}(\gamma /s)1_{[0,\gamma \mathrm{e}]}(s)ds,
		\] 
		where $f(u)=u\csc u \exp (-u\cot u)$.
		In case of $\gamma=1$, 
		for the shape of the density of 
		$\nu_{\mathfrak{y}_{1}}$,
		 see the graph below.\\
	\includegraphics{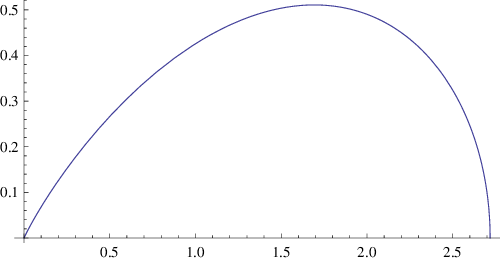}\\
	\item[\rm{(5)}]
	It holds the following formulas:
	\begin{align*}
	&\mathfrak{y}_{\gamma}^{\boxplus t}
	= D_{t}(\mathfrak{y}_{\gamma}^{\boxtimes 1/t}),&\\
	&\mathfrak{y}_{\gamma}^{\boxtimes t} 
	= D_{t}(\mathfrak{y}_{\gamma}^{\boxplus 1/t}).&
	\end{align*}
\end{enumerate}
\end{thm}
The proof of this theorem is helped by 
the following property and the well-known integral representation of 
the Lambert's $W$--function (for instance, see \cite[Section 4]{Cetc96} and \cite[Theorem 3.1]{KJCB11}):
\begin{prop}\label{int}
\begin{enumerate}[{\rm (1)}]
\item
The principal branch of $W_{0}(z)$ has an analytic extension on $\mathbb{C} \backslash (-\infty, -1/e]$ 
and it takes $\C^- \cup \R$ on $\C^{-}$.  
\item
For any $z \in \C^{+}$, we have an integral representation:
\[
\frac{W_{0}(z)}{z}
=\dfrac{1}{\, \pi \, }
\int_{0}^{\pi}\frac{(1-u\cot u)^2+u^2}{z+u\csc u \exp (-u\cot u)}du.
\]
\end{enumerate}
\end{prop}
\begin{proof}[Proof of Theorem $11$]
\begin{enumerate}[{\rm (1)}]
\item
The $\boxtimes$--infinitely divisibility is trivial from the form of the $S$--transform and the facts in \cite{BeVo92}.
By Proposition \ref{Scum}, we have 
\[
R_{\mathfrak{y}_{\gamma}}(z\mathrm{e}^{-\gamma z})=z.
\]
Then the $R$--transform is given by using the Lambert's $W$--function as follows:
\begin{align*}
R_{\mathfrak{y}_{\gamma}}(z)=-\frac{1}{\gamma}W_{0}(-\gamma z), 
\end{align*}
and hence,
\begin{align*}
\mathcal{R}_{\mathfrak{y}_{\gamma}}(z)=\frac{W_{0}(-\gamma z)}{-\gamma z}. 
\end{align*}
\item
By Proposition \ref{id} and Proposition \ref{int} (1), 
$\mathcal{R}_{\mathfrak{y}_{\gamma}}$ has an analytic extension defined on $\C^{-}$ with value $\C^- \cup \R$,
which means that $\mathfrak{y}_{\gamma}$ is $\boxplus$--infinitely divisible.

\item
The Taylor type expansion of $-W_{0}(-z)$ at the origin 
is obtained from Equation (3.1) of \cite[pp. 339]{Cetc96}.

\item 
We put
$g(u)=(1-u\cot (u))^2+u^2$.
Noting that
\begin{align}\label{g}
g(u) = \frac{uf'(u)}{f(u)},
\end{align}
we obtain
\begin{align*}
\mathcal{R}_{\mathfrak{y}_{\gamma}}(z)
&= \frac{1}{\pi}\int_{0}^{\pi}\frac{g(u)}{-\gamma z+f(u)}{\mathrm d}u&\\
&= \frac{1}{\pi}\int_{0}^{\pi}\frac{g(u)/f(u)}{1-\gamma z/f(u)}{\mathrm d}u
= \frac{1}{\gamma \pi}\int_{0}^{\mathrm{\gamma e}}\frac{f^{-1}(\gamma/s)}{1-sz}
{\mathrm d}s,&
\end{align*}
where we have changed the variables as $s=\gamma/f(u)$.
\begin{align*}
\mathcal{R}_{\mathfrak{y}_{\gamma}}(z)
&= \frac{1}{\gamma \pi}\int_{0}^{\mathrm{\gamma e}}\frac{f^{-1}(\gamma/s)}{1-sz}{\mathrm d}s&\\
&= \frac{1}{\gamma \pi}
\int_{0}^{\mathrm{\gamma e}}
\left(
\frac{sz}{1-sz}+1
\right)
f^{-1}(\gamma/s){\mathrm d}s&\\
&= \frac{1}{\gamma}+ \frac{1}{\gamma \pi}
\int_{0}^{\mathrm{\gamma e}}
\frac{z}{1-sz}
sf^{-1}(\gamma/s){\mathrm d}s.&
\end{align*}
Therefore we obtain the L\'evy measure 
$\nu_{\mathfrak{y}_{\gamma}}({\mathrm d}s)=\dfrac{sf^{-1}(\gamma /s)}{\gamma\pi}{\mathrm d}s$ of $\mathfrak{y}_{\gamma}$.

\item It is direct consequence of Proposition \ref{basic}.
\end{enumerate}
\end{proof}

	\begin{rem}
	Here we consider the limit distribution with parameter $\gamma=1$.
	For example, if $\rho$ is the free Poisson distribution 
	with parameter $1$, this is the case.
	Simply we write $\mathfrak{y}$ 
	instead of $\mathfrak{y}_{1}$.
	There exists a probability measure $\rho$ such that 
	\begin{align}\label{yr}
	\mathcal{R}_{\rho}(z)
	=\dfrac{\mathcal{R}_{\mathfrak{y}}(z)-1}{z}.
	\end{align}

	Indeed,
	if we consider the shifted free cumulant sequence
	$\{k_{n}(\rho)\}_{n\in\N}=\left\{\frac{(n+1)^{n}}{(n+1)!}\right\}_{n\in\N}$,
	which is a sequence 
	of coefficients of Taylor expansion of $R_{\rho}$ at 0,  
	then it becomes a moment sequence of a probability measure.
	This means that the measure $\rho$ 
	is a free compound Poisson distribution
	with the compound measure $\sigma$, the moments of which are 
	$m_{n}(\sigma)=\frac{(n+1)^{n}}{(n+1)!}$. 
	From \eqref{yr}, we have
	\begin{align}\label{rf}
	z \mathcal{R}_{\mathfrak{y}}(zM_{\rho}(z))=zM_{\rho}(z).
	\end{align}
	By putting $P(z) = zM_{\rho}(z)$
	and using the Lagrange inversion formula, \eqref{rf} implies that
	\[
	n^{\text{th}}\text{ coefficient of } \{P(z)\}
	=
	\dfrac{1}{\,n\,}\times
	\left(
	(n-1)^{\text{st}}\text{ coefficient of } 
	\mathcal{R}_{\rho}(z)
	\right).
	\]
	Hence we obtain 
	the moments of $\rho$ 
	as 
	$$
	m_{n}(\rho)=\frac{(2n+1)^{n-1}}{n!}.
	$$
	\end{rem}

%%%%%%Boolean case%%%%%%%%%%%%%%%%%%%	

\subsection{On the limit distribution in boolean case}
Let $\mfs:=\mfs_{1}$ denote a probability measure with the moment sequence 
$\Big\{ \displaystyle{ \frac{n^n}{\, n! \, } } \Big\}_{n \ge 0}$,
the positivity of which is ensured by \cite{Ml10}.
Then its moment generating function $M_{\mfs}(z)$ can be given by 
$$
  M_{\mfs}(z) = \sum_{n = 0}^{\infty} \frac{n^n}{\, n! \, } \, z^n 
            = \frac{1}{1 - \eta (z)}
\eqno{(1)}
$$
where the function $\eta(z)$ is defined by 
$$
  \eta (z) = - W_0 (-z), \quad 
  z \in {\mathbb C} \backslash  \Big[ \, \frac{1}{\, e \,}, \infty \Big).
$$

\bigskip 

\bigskip 

\noindent
\begin{rem}
The following useful facts on the function $\eta$ can be found 
in \cite[Sect.2]{DyHa04}:
The map 
$$
 \theta \longmapsto \frac{\sin \theta}{\theta} 
                    \exp \big( \theta \cot \theta \big)
$$
is a bijection of $(0, \pi)$ onto $(0, e)$, and if we 
define the functions 
$\eta^{+}, \eta^{-}  \, : \, 
\big[ \, \mbox{\Large $\frac{1}{\, e \,}$}, \infty \big) 
\rightarrow {\mathbb C}$ by 
$$
  \eta^{\pm} \Big(  
                   \frac{\theta}{\sin \theta}
                   \exp \big( - \theta \cot \theta \big) \Big)
  = \theta \cot \theta \pm i \, \theta,
   \quad 0 \le \theta < \pi,
$$
then 
$$
    \eta^{\pm} (x) = \lim_{y \, \downarrow \, 0} \eta (x + i y),
    \quad x \in \big[ \, \mbox{\Large $\frac{1}{\, e \,}$}, \infty \big).
$$
\end{rem}

From (1), the Cauchy transform of the measure $\mfs$ is given by 
$$
   G_{\mfs} (\zeta) = 
    \frac{1}{\, \zeta \,} \, 
    \frac{1}{1 - \eta \Big( \mbox{\Large $\frac{1}{\, \zeta \,}$} \Big)},
\quad \mbox{ for } \zeta \in {\mathbb C} \setminus [0, e].
$$

Now we apply the Stieltjes inversion formula to obtain the 
density function $\varphi_{\mfs} (t)$ of the measure $\mfs$, that is, 
for $t \in [ 0, e ]$, 
$$
 \begin{aligned}
  \varphi_{\mfs} (t) 
     & = - \frac{1}{\, \pi \,} 
         \lim_{\varepsilon \, \downarrow \, 0} 
         \mathrm{Im} \Big( G_{\mfs} (t + i \varepsilon) \Big) \\
     & = - \frac{1}{\, \pi \,} 
         \lim_{\varepsilon \, \downarrow \, 0} 
         \mathrm{Im} \Bigg(
           \frac{1}{\, t + i \varepsilon \,} \, 
           \frac{1}{1 - \eta \Big( 
                     \mbox{\Large $\frac{1}{\, t + i \varepsilon \,}$} 
                    \Big)}\Bigg) \\
     & = - \frac{1}{\, \pi \,} 
         \lim_{\varepsilon \, \downarrow \, 0} 
         \mathrm{Im} \Bigg(
           \frac{\, t - i \varepsilon \,}{\, t^2 + \varepsilon^2 \,} \, 
           \frac{1}{1 - \eta \Big( 
                     \mbox{\Large $\frac{\, t - i \varepsilon \,}
                                        {\, t^2 + \varepsilon^2 \,}$} 
                    \Big)}\Bigg) \\
     & = - \frac{1}{\, \pi \,} 
         \mathrm{Im} \Bigg(
           \frac{1}{\, t \,} \, 
           \frac{1}{1 - \eta^{-} \Big( 
                     \mbox{\Large $\frac{1}{\, t\,}$} 
                     \Big)}\Bigg) ,
 \end{aligned}
$$
where the function $\eta^{-}$ is defined as in Remark above.
Here we change the variables 
$$
   \frac{1}{\, t\,} 
 = \frac{\theta}{\sin \theta} \exp \big( - \theta \cot \theta \big),
$$
then it follows that
$$
 \begin{aligned}
  \varphi_{\mfs} (t) 
     &= - \frac{1}{\, \pi \,} 
         \mathrm{Im} \Bigg(
           \frac{1}{\, t \,} \, 
           \frac{1}{1 - \big( \theta \cot \theta  -  i \theta \big)}\Bigg) \\
     &=   \frac{1}{\, \pi \,} 
           \frac{1}{\, t \,} \, 
           \frac{\theta}
           {\big(1 - \theta \cot \theta \big)^2  + \theta^2} \\
     &= \frac{1}{\, \pi \,} 
       \Big( 
       \frac{\theta}{\sin \theta} \exp \big( - \theta \cot \theta \big) 
       \Big) \, 
       \Bigg( \frac{ \theta}
              {\big(1 - \theta \cot \theta \big)^2  + \theta^2}\Bigg) \\
     &= \frac{1}{\, \pi \,} 
       \frac{\theta^2 \exp \big( - \theta \cot \theta \big) }
       {\sin \theta 
          \Big( \big(1 - \theta \cot \theta \big)^2  + \theta^2 \Big)} .
 \end{aligned}
$$
Thus we obtain the following proposition:
\begin{prop}
The probability density function $\varphi_{\mfs}$ of the measure $\mfs$ 
can be given by the implicit (parametric) form as
\[
   \varphi_\mfs \Big( \frac{\sin v}{v} \exp \big( v \cot v \big) \Big)
 = \frac{1}{\, \pi \,} 
       \frac{v^2 \exp \big( - v \cot v \big) }
       {\sin v 
          \Big( \big(1 - v \cot v \big)^2  + v^2 \Big)},
 \qquad 0 < v < \pi.
\]
\end{prop}
\begin{rem}
\begin{itemize}
\item[(1)]{The shape of the density function of $\varphi_{\mfs}$ is as the graph below, especially non-unimodal.}\\
\includegraphics{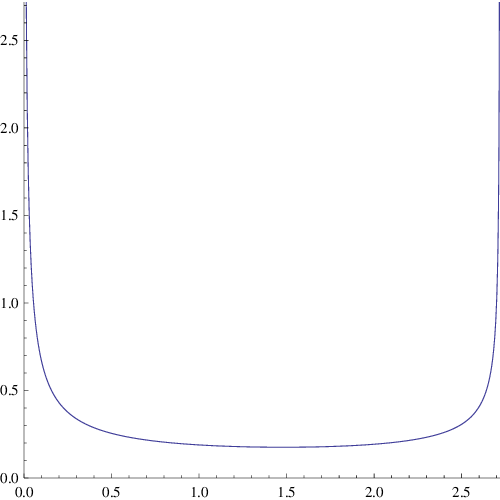}
\item[(2)]{The function $\big(1 - v \cot v \big)^2  + v^2$ also 
appears in the integral representation of $W_0(z)/z$ as we mentioned Proposition \ref{int}:
$$
 \frac{W_0(z)}{z} = \frac{1}{\, \pi \,} \! \int_{0}^{\pi} 
        \frac{\big(1 - v \cot v \big)^2  + v^2 }{z + 
            \mbox{\Large $\frac{v}{\sin v}$} \exp \big( - v \cot v \big)} 
      \mathrm{d} v.
$$
Thus
using $f(v)=v \csc v \exp ( - v \cot v )$ and \eqref{g} again,
the parametric form of the density function can be rewritten as 
$$
  \varphi_{\mfs} \Big( \frac{1}{\, f(v) \,} \Big) 
 = \frac{1}{\, \pi \,} \frac{\big( f(v) \big)^2}{ f'(v) }.
$$
}
\end{itemize}
\end{rem}
%%%%%%%%%%%%%%%%%%%%%%%%%%%%%%%%%%%%%%%%%%%%%%%%
	\subsection{Concluding remark}
	We also know that there exists the similar moment sequence.
	In the paper by Dykema and Haagerup \cite{DyHa04}, 
	they find a limit distribution of DT-operator DT($1,\delta_{0}$).
	The moment of their one is $m_{n}=\frac{n^{n}}{(n+1)!}$.
	A natural question arises: how do we realize this limit theorem via random matrix model?

\end{document}